\documentclass[11pt]{amsart}
\usepackage{amscd,amsmath,amsthm,amssymb}
\usepackage{pstcol,pst-plot,pst-3d}
\usepackage{lmodern,pst-node}
\def\frk{\frak}               

\def\Phi{{\frk n}}
\def\Phi{{\frk N}}
\def\Lc{{\mathcal L}}
\def\Gc{{\mathcal G}}
\def\Sc{{\mathcal S}}
\def\opn#1#2{\def#1{\operatorname{#2}}} 
\opn\chara{char} \opn\length{\ell} \opn\pd{pd} \opn\rk{rk}
\opn\projdim{proj\,dim} \opn\injdim{inj\,dim} \opn\rank{rank}
\opn\depth{depth} \opn\grade{grade} \opn\height{height}
\opn\embdim{emb\,dim} \opn\codim{codim}

\opn\Tr{Tr} \opn\bigrank{big\,rank}
\opn\superheight{superheight}\opn\lcm{lcm}
\opn\trdeg{tr\,deg}
\opn\reg{reg} \opn\lreg{lreg} \opn\ini{in} \opn\lpd{lpd}
\opn\size{size} \opn\sdepth{sdepth}
\opn\link{link}\opn\fdepth{fdepth}\opn\lex{lex}
\opn\div{div} \opn\Div{Div} \opn\cl{cl} \opn\Cl{Cl}
\opn\Spec{Spec} \opn\Supp{Supp} \opn\supp{supp} \opn\Sing{Sing}
\opn\Ass{Ass} \opn\Min{Min}\opn\Mon{Mon}
\opn\Ann{Ann} \opn\Rad{Rad} \opn\Soc{Soc}
\opn\Im{Im} \opn\Ker{Ker} \opn\Coker{Coker} \opn\Am{Am}
\opn\Hom{Hom} \opn\Tor{Tor} \opn\Ext{Ext} \opn\End{End}
\opn\Aut{Aut} \opn\id{id}

\opn\nat{nat}
\opn\pff{pf}
\opn\Pf{Pf} \opn\GL{GL} \opn\SL{SL} \opn\mod{mod} \opn\ord{ord}
\opn\Gin{Gin} \opn\Hilb{Hilb}\opn\sort{sort}
\opn\aff{aff} \opn\con{conv} \opn\relint{relint} \opn\st{st}
\opn\lk{lk} \opn\cn{cn} \opn\core{core} \opn\vol{vol}
\opn\link{link} \opn\star{star}\opn\lex{lex}\opn\set{set}
\opn\gr{gr}

\def\pot#1#2{#1[\kern-0.28ex[#2]\kern-0.28ex]}

\opn\dirlim{\underrightarrow{\lim}}
\opn\inivlim{\underleftarrow{\lim}}
\let\union=\cup
\let\sect=\cap

\let\iso=\cong
\let\Union=\bigcup

\let\to=\rightarrow

\def\Implies{\ifmmode\Longrightarrow \else
        \unskip${}\Longrightarrow{}$\ignorespaces\fi}
\def\implies{\ifmmode\Rightarrow \else
        \unskip${}\Rightarrow{}$\ignorespaces\fi}
\def\iff{\ifmmode\Longleftrightarrow \else
        \unskip${}\Longleftrightarrow{}$\ignorespaces\fi}

\let\:=\colon
\newtheorem{Theorem}{Theorem}[section]
\newtheorem{Lemma}[Theorem]{Lemma}
\newtheorem{Corollary}[Theorem]{Corollary}
\newtheorem{Proposition}[Theorem]{Proposition}

\newtheorem{Example}[Theorem]{Example}

\let\epsilon\varepsilon
\let\kappa=\varkappa
\def\qed{\ifhmode\textqed\fi
      \ifmmode\ifinner\quad\qedsymbol\else\dispqed\fi\fi}
\def\textqed{\unskip\nobreak\penalty50
       \hskip2em\hbox{}\nobreak\hfil\qedsymbol
       \parfillskip=0pt \finalhyphendemerits=0}
\def\dispqed{\rlap{\qquad\qedsymbol}}

\opn\dis{dis}
\def\pnt{{\raise0.5mm\hbox{\large\bf.}}}

\opn\Lex{Lex}

\numberwithin{equation}{section}

\begin{document}

\title {Determinantal Facet Ideals}

\author {Viviana Ene, J\"urgen Herzog, Takayuki Hibi  and Fatemeh Mohammadi}

\address{Viviana Ene, Faculty of Mathematics and Computer Science, Ovidius University, Bd.\ Mamaia 124,
 900527 Constanta, Romania
 \newline \indent Institute of Mathematics of the Romanian Academy, P.O. Box 1-764, RO-014700, Buchaest, Romania} \email{vivian@univ-ovidius.ro}

\address{J\"urgen Herzog, Fachbereich Mathematik, Universit\"at Duisburg-Essen, Campus Essen, 45117
Essen, Germany} \email{juergen.herzog@uni-essen.de}

\address{Takayuki Hibi, Department of Pure and Applied Mathematics, Graduate School of Information Science and Technology,
Osaka University, Toyonaka, Osaka 560-0043, Japan}
\email{hibi@math.sci.osaka-u.ac.jp}

\address{Fatemeh Mohammadi, Mathematical Sciences Research Institute, Berkeley, CA 94720-3840, U.S.A.}\email{fatemeh.mohammadi716@gmail.com}

\begin{abstract}
We consider ideals generated by general sets of $m$-minors of an $m\times n$-matrix of indeterminates. The generators are identified with the facets of an $(m-1)$-dimensional pure simplicial complex. The ideal generated by the minors corresponding to the facets of such a complex is called a determinantal facet ideal. Given a pure simplicial complex $\Delta$,  we discuss the question when the generating minors of its determinantal facet ideal $J_\Delta$  form a Gr\"obner basis and when $J_\Delta$ is a prime ideal.
\end{abstract}

\thanks{
This paper was written while J.\ Herzog and F.\ Mohammadi were staying
at Osaka University.  They were supported by the JST (Japan Science
and Technology Agency) CREST (Core Research for Evolutional Science
and Technology) research project {\em Harmony of Gr\"obner Bases and
the Modern Industrial Society} in the frame of the Mathematics Program
``Alliance for Breakthrough between Mathematics and Sciences''.\\
 The first author was supported by the grant UEFISCDI,  PN-II-ID-PCE- 2011-3-1023.}
\maketitle

\section*{Introduction}
Let $K$ be a field,  $X=(x_{ij})$ be an $m\times n$-matrix of indeterminates and $S=K[X]$ be the polynomial ring over $K$ in the indeterminates $x_{ij}$. We assume that $m\leq n$. Classically the ideals $I_t(X)$ generated by all $t$-minors of $X$ have been considered. Hochster  and  Eagon \cite{Ho} proved that the  rings $S/I_t(X)$ are normal Cohen--Macaulay domains. A standard reference  on the classical theory of determinantal ideals, including the study of the powers of $I_t(X)$ is the book \cite{BV} of Bruns and Vetter. Motivated by geometrical considerations the more general class of ladder determinantal ideals have been considered as well, \cite{A}. A new aspect to the theory of determinantal ideals was introduced by Sturmfels \cite{St90} and Caniglia et al. \cite{CGG} who showed that the $t$-minors of $X$ form a Gr\"obner basis of $I_t(X)$ with respect to any monomial order which selects the diagonals of the minors as leading terms. This technique provides a new proof of the Cohen--Macaulayness of the determinantal rings $S/I_t(X)$ and was subsequently also used to compute important numerical invariants of these rings, including the $a$-invariant, the multiplicity and the Hilbert function, see \cite{BH1}, \cite{CH} and \cite{HT}. An excellent survey on the  theory of determinantal ideals regarding the Gr\"obner basis aspect and with many references to more recent work is the article \cite{BC} by Bruns  and Conca.

Applications in algebraic statistics prompted the study of determinantal ideals generated by quite general classes of minors, including ideals generated by adjacent 2-minors, see \cite{HSS} and \cite{HH}, or ideals generated by an arbitrary set of 2-minors  in a $2 \times n$-matrix \cite{HHH}. Thus one may raise the following questions: given an arbitrary set of minors of $X$,  what can be said about the ideal they generate? When is such an ideal a radical ideal, when is it a prime ideal, what is its primary decomposition, when is it Cohen--Macaulay, what is its Gr\"obner basis? Apart from the classical cases mentioned before, satisfying answers to some of these questions are known for ideals generated by arbitrary sets of $2$-minors of a $2\times n$-matrix of indeterminates. All these ideals are radical, their primary decomposition and their Gr\"obner basis are known, see \cite{HHH}.

The purpose of this paper is to extend some of the results shown in \cite{HHH} to ideals generated by an arbitrary set of maximal minors of an $m\times n$-matrix of indeterminates. For any sequence of integers $1\leq a_1<a_2<\cdots<a_m\leq n$ we denote by $[a_1a_2 \ldots a_m]$ the maximal minor of $X$ with columns $a_1,a_2,\ldots,a_m$. The set of integers $\{a_1,a_2,\ldots,a_m\}$ may be viewed as a facet of a simplex on the vertex set $[n]$. This leads us to the following definition: let $\Delta$ be a pure simplicial complex  on the vertex set $[n]=\{1,\ldots,n\}$ of dimension $m-1$. With each facet $F=\{a_1<a_2<\cdots<a_m\}$ we associate the minor $\mu_F= [a_1a_2 \ldots a_m]$, and call the ideal
\[
J_\Delta=(\mu_F\:\; F\in \mathcal{F}(\Delta))
\]
the {\em determinantal facet ideal} of $\Delta$. Here  $\mathcal{F}(\Delta)$ denotes the set of facets of $\Delta$.

When $m=2$,  $\Delta$ may be identified with a graph $G$ and the $m$-minors are binomials. In that case the determinantal facet ideal coincides with the binomial edge ideal of \cite{HHH}.

In the first section of this paper we answer the question when the maximal minors generating $J_\Delta$ form a Gr\"obner basis of $J_\Delta$. In order to explain this result, we have to introduce some notation. Let $\Gamma$ be a simplicial complex. We denote by $\Gamma^{(i)}$ the $i$-skeleton of $\Gamma$. The simplicial complex  $\Gamma^{(i)}$ is the collection of all simplices of $\Gamma$ whose dimension is at most $i$.

Now let $\Delta$ be a pure $(m-1)$-dimensional  simplicial complex on the vertex set $[n]=\{1,2,\ldots,n\}$. We denote by $\Sc$ the set of  simplices $\Gamma$ with vertices in $[n]$ with $\dim \Gamma\geq m-1$   such that $\Gamma^{(m-1)}\subset \Delta$.   Let $\Gamma_1,\ldots,\Gamma_r$ be the maximal elements in $\Sc$  (with respect to inclusion) and set $\Delta_i=\Gamma_i^{(m-1)}$. Then $\Delta =\Delta_1\union \Delta_2\union\cdots\union \Delta_r$. The simplicial complex whose facets are the   $\Gamma_i$ is called the {\em clique complex} of $\Delta$, the $\Delta_i$ are called the {\em cliques} of $\Delta$ and   $\Delta =\Delta_1\union \Delta_2\union\cdots\union \Delta_r$ the  {\em clique decomposition} of $\Delta$. For example, let $\Delta$ be the $2$-dimensional simplicial complex on the vertex set $[7]$ with the facets
$F_1=\{1,2,3\}, F_2=\{1,2,4\}, F_3=\{1,3,4\}, F_4=\{2,3,4\}, F_5=\{3,4,5\},$ and $F_6=\{5,6,7\}$. Then $\Delta$ has the clique decomposition
$\Delta=\Delta_1\cup\Delta_2\cup\Delta_3$ with $\Delta_1=\Gamma_1^{(2)}$, where $\Gamma_1$ is the $3$-dimensional simplex on the set $[4],$
$\Delta_2=\Gamma_2^{(2)},$ where $\Gamma_2$ is the $2$-dimensional simplex on the set $\{3,4,5\}$, and $\Delta_3=\Gamma_3^{(2)},$ where
$\Gamma_3$ is the $2$-dimensional simplex on the set $\{5,6,7\}.$

 Note that if $m=2,$ that is, $\Delta$ is a graph, then $\Delta_i$ are exactly the cliques of $\Delta$ as they are known in  graph theory and $\Gamma_1,\ldots,\Gamma_r$ are the facets of the clique complex of the graph $\Delta.$

The complex $\Delta$ is called {\em closed} (with respect to the given labeling) if for any two facets $F=\{a_1<\cdots <a_m\}$ and $G=\{b_1<\cdots <b_m\}$ with $a_i=b_i$ for some $i$, the $(m-1)$-skeleton of the simplex on the vertex set $F\cup G$ is contained in $\Delta$. In terms of its clique decomposition, the property of $\Delta$ of being closed can be expressed in the following ways:

\medskip
(1)  $\Delta$ is closed,  if and only if for all $i\neq j$ and all $F=\{a_1<a_2<\cdots <a_m\} \in \Delta_i$ and $G=\{b_1<b_2<\cdots<b_m\}\in \Delta_j$
we have $a_\ell\neq b_\ell$ for all $\ell$.

\medskip
(2) $\Delta$ is closed,  if and only if for all $i\neq j$ and all $\{a_1,\ldots, a_m\}\in\Delta_i$ and  $\{b_1,\ldots, b_m\}\in\Delta_j$, the monomials ${\ini}_<[a_1\ldots a_m]$ and
${\ini}_<[b_1\ldots b_m]$ are relatively prime, where $<$ is the lexicographical order induced by the natural order of indeterminates
\[
x_{11}>x_{12}>\cdots >x_{1n}>x_{21}>\cdots >x_{2n}>\cdots >x_{mn},
\]
row by row from left to right.

The main result (Theorem~\ref{closedfatemeh}) of Section~1 states that the minors generating the facet ideal $J_\Delta$ form a quadratic Gr\"obner basis with respect to the lexicographic order induced by the natural order of the variables, if and only if $\Delta$ is closed. We also show that whenever $\Delta$ is closed, then $J_\Delta$ is Cohen--Macaulay and the $K$-algebra generated by the minors which generate $J_\Delta$ is Gorenstein, see Corollary~\ref{cm} and Corollary~\ref{gorenstein}.

In Section~2 we discuss when a determinantal facet ideal is a prime ideal. As a main result we show in Theorem~\ref{viviana} that if $\Delta$ is closed and $J_\Delta$ is a prime ideal, then the clique complexes $\Delta_i$ of $\Delta$ satisfy the following intersection properties: for all $2\leq t\leq m=\dim\Delta+1$ and for any pairwise distinct cliques $\Delta_{i_1},\ldots,\Delta_{i_t}$ one has
\[
|V(\Delta_{i_1})\cap\cdots \cap V(\Delta_{i_t})|\leq m-t.
\]
We expect that this intersection property actually characterizes  closed simplicial complexes whose determinantal facet ideal is prime, but could not prove it yet.  In Theorem~\ref{fatemehinverse} we can give only a partial converse of Theorem~\ref{viviana}.

We show in Example~\ref{necessary} that primality of  determinantal facet ideals satisfying the above intersection condition can only be expected for closed simplicial complexes. For non-closed simplicial complexes the primality problem seems to be pretty hard.

In Section 3 we study primality of $J_\Delta$ for a closed simplicial complex under the following very strict intersection condition: let $\Delta=\Delta_1\cup\ldots\cup\Delta_r$ be the clique decomposition of $\Delta$. We require  that
\begin{enumerate}
\item[(i)] $|V(\Delta_i)\sect V(\Delta_j)|\leq 1$ for all $i< j$;
\item[(ii)] $V(\Delta_i)\cap V(\Delta_j)\cap V(\Delta_k)=\emptyset$ for all $i<j<k$.
\end{enumerate}
For $m=3$, this is exactly the necessary condition for primality formulated in Theorem~\ref{viviana}.

Assuming (i) and (ii), we let $G_\Delta$ be the simple graph  with the vertices $v_1,\ldots,v_r$, and
edges $\{v_i,v_j\}$ for all $i\neq j$ with  $V(\Delta_i)\sect V(\Delta_j)\neq \emptyset$.  The question arises for which graphs $G_\Delta$ the determinantal facet ideal $J_\Delta$ is a prime ideal. This is the case when  $\Delta$ is closed and $G_\Delta$ is a forest or a cycle, see Theorem~\ref{tree} and Theorem~\ref{cycle}.
Finally we show in Theorem~\ref{planar} that   for any graph $G$ there is a closed simplicial complex $\Delta$  with $G=G_\Delta$ whose  cliques are all simplices.

\section{Determinantal facet ideals whose generators form a Gr\"obner basis}

In this section we intend to classify those ideals generated by maximal minors of a generic $m\times n$-matrix $X$ whose generating minors form a Gr\"obner basis. As explained in the introduction we identify each $m$-minor $[a_1a_2 \ldots a_m]$ of $X$ with the $(m-1)$-simplex $F=\{a_1,a_2,\ldots,a_m\}$. Thus an arbitrary collection of $m$-minors of $X$ can be indexed by the facets of a pure $(m-1)$-dimensional simplicial complex $\Delta$ on the vertex set $[n]$. The ideal generated by these minors will be denoted $J_\Delta$, and is called the {\em determinantal facet ideal} of $\Delta$. In other words, if $\mathcal{F}(\Delta)$ denotes the set of facets of $\Delta$, then $J_\Delta=(\mu_F\:\; F\in \mathcal{F}(\Delta))$,  where  $\mu_F=[a_1a_2\cdots a_m]$ for $F=\{a_1,a_2,\ldots,a_m\}$.

In analogy to the case of 2-minors, as considered in \cite{HHH}, we say that  $\Delta$ is  {\em closed with respect to the given labeling} if for any two facets $F=\{a_1<\cdots <a_m\}$ and $G=\{b_1<\cdots <b_m\}$ with $a_i=b_i$ for some $i$, the $(m-1)$-skeleton of the simplex on the vertex set $F\cup G$ is contained in $\Delta$. $\Delta$ is called {\em closed} if there is a labeling of its vertices such that $\Delta$ is closed with respect to it.

For example, let $\Delta$ be the $2$-dimensional simplicial complex of Figure~\ref{Fig1} (a). The cliques of $\Delta$ are two simplices of
dimension $2.$.

\begin{figure}[h]
\begin{center}
\psset{unit=0.5cm}
\begin{pspicture}(-8.3,-2.5)(4,3.5)

\psline(-12,2)(-12,-2)
\psline(-12,-2)(-10,0)
\psline(-10,0)(-12,2)
\psline(-10,0)(-8,-2)
\psline(-8,-2)(-8,2)
\psline(-10,0)(-8,2)
\rput(-10,-3){(a)}

\psline(-4,2)(-4,-2)
\psline(-4,-2)(-2,0)
\psline(-2,0)(-4,2)
\psline(-2,0)(0,-2)
\psline(0,-2)(0,2)
\psline(-2,0)(0,2)
\rput(-2,-3){(b)}
\rput(-4.5,-2){$1$}
\rput(-4.5,2){$2$}
\rput(-2,-0.7){$3$}
\rput(0.5,-2){$5$}
\rput(0.5,2){$4$}

\psline(4,2)(4,-2)
\psline(4,-2)(6,0)
\psline(6,0)(4,2)
\psline(6,0)(8,-2)
\psline(8,-2)(8,2)
\psline(6,0)(8,2)
\rput(6,-3){(c)}
\rput(3.5,-2){$2$}
\rput(3.5,2){$1$}
\rput(6,-0.7){$5$}
\rput(8.5,-2){$4$}
\rput(8.5,2){$3$}

\end{pspicture}
\end{center}
\caption{}\label{Fig1}
\end{figure}
 $\Delta$ is closed with respect to the labeling given in Fi\-gure~\ref{Fig1} (b), but it is not closed
with respect to the labeling given in Figure~\ref{Fig1} (c). Indeed, with respect to the first labeling, the facets $\{1,2,3\}$ of the first clique
and $\{3,4,5\}$ of the second clique do not have a common label in the same position, while with respect to the second labeling the facets
$\{1,2,5\}$ and $\{3,4,5\}$ have the label $5$ in the last position. In terms of initial monomials, in the first case
$\ini_<[123]=x_{11}x_{22}x_{33}$ and $\ini_<[345]=x_{13}x_{24}x_{35}$ are relatively prime, while in the second case $\ini_<[125]=x_{11}x_{22}x_{35}$
and $\ini_<[345]=x_{13}x_{24}x_{35}$ are not relatively prime. However the simplicial complex is closed since one may find a labeling of its vertices
with respect to which $\Delta$ is closed.

\medskip
The main result of this section is

\begin{Theorem}
\label{closedfatemeh}
The set $\Gc=\{[a_1\ldots a_m]:\; \{a_1,\ldots,a_m\}\in \Delta\}$ is a Gr\"obner basis of $J_\Delta$ with respect to the lexicographical order induced by the natural order of indeterminates if and only if $\Delta$ is closed.
\end{Theorem}

Before proving the above theorem, we introduce the following notation which is often used in the classical determinantal ideal theory.
If $r<m,$ the minor corresponding to the submatrix of $X$ with rows $a_1,\ldots,a_r$ and columns $b_1,\ldots,b_r$ is denoted by
$[a_1\ldots a_r| b_1\ldots b_r].$
For the proof of Theorem~\ref{closedfatemeh} we need the following technical result.

\begin{Lemma}
\label{det}
Let $m\leq n-1.$ For any $m-1$ rows $c_1,c_2,\ldots,c_{m-1}$ and $m+1$ columns $d_1,d_2,\ldots,d_{m-2},e_1, e_2,e_3$ of $X$ one has the following identity:
\begin{eqnarray*}
(-1)^{k}[c_1\ldots c_{m-1}| d_1\ldots d_{m-2}e_3][d_1\ldots d_{m-2}e_1e_2]&\\
+(-1)^j[c_1\ldots c_{m-1}| d_1\ldots d_{m-2}e_2][d_1\ldots d_{m-2}e_1e_3]&\\
+(-1)^i[c_1\ldots c_{m-1}| d_1\ldots d_{m-2}e_1][d_1\ldots d_{m-2}e_2e_3]=0,
\end{eqnarray*}
provided  that $d_1<d_2<\cdots< d_{i-1}<e_1<d_{i}<\cdots <d_{j-2}<e_2<d_{j-1}<\cdots<d_{k-3}<e_3<d_{k-2}<\cdots<d_{m-2}$ for some $1\leq i<j<k\leq m.$
\end{Lemma}

\begin{proof}
Our assumption on the sequence of integers means that $e_1$ is the $i^{\rm th}$ term,
$e_2$   the $j^{\rm th}$ term, and $e_3$  the $k^{\rm th}$ term of the above sequence.

Now consider the matrix
\[
M=\left(
   \begin{array}{cccccccccc}
     x_{1 d_1}& \ldots & x_{1 d_{i-1}}& x_{1 e_1} &\ldots& x_{1 e_2}& \ldots & x_{1 e_3}& \ldots &x_{1 d_{m-2}} \\
          \vdots &  & \vdots & \vdots & & \vdots  & & \vdots & & \vdots         \\
 x_{m d_1}& \ldots & x_{m d_{i-1}}& x_{m e_1} &\ldots& x_{m e_2}& \ldots & x_{m e_3}& \ldots &x_{m d_{m-2}} \\
 g_{d_1} & \ldots & g_{d_{i-1}} & g_{e_1} &\ldots& g_{e_2}& \ldots & g_{e_3}& \ldots &g_{d_{m-2}} \\
        \end{array}
 \right),
\]
where $g_\ell$ is the minor $[c_1\ldots c_{m-1}|d_1\ldots d_{m-2} \ell]$ of $X$ for each  $\ell\in \{d_1,d_2,\ldots, d_{m-1},e_1,e_2,e_3\}$. Expanding $g_\ell$ by
the last column we get
$$g_\ell=\sum_{i=1}^{m-1}(-1)^{m-1+i}[c_1\ldots c_{{i-1}}c_{i+1}\ldots c_{m-1}|d_1\ldots d_{m-2}] x_{c_i\ell}$$ for each $\ell$. Therefore  the last row of $M$
is a linear combination of the rows $c_1,\ldots,c_{m-1}$ of $M$. Hence  the determinant of $M$ is zero. On the other hand, $g_\ell=0$ for $\ell=d_1,\ldots,d_{m-2}$, since for these $\ell$ the polynomial $g_\ell$ is the determinant of a matrix with two equal columns. Now computing the determinant of $M$ by expanding its last row we obtain the desired identity.
\end{proof}

\begin{proof}[Proof of Theorem~\ref{closedfatemeh}]
Assume that $\Delta$ is closed. We show that all $S$-pairs, $$S([a_1\ldots a_m],[b_1\ldots b_m])$$ reduce
to zero. If $a_i\neq b_i$ for all $i$, then  $\text{in}_<[a_1\ldots a_m]$ and $\text{in}_<[b_1\ldots b_m]$  have no common factor. Therefore $S([a_1\ldots a_m],[b_1\ldots b_m])$ reduces to zero.

Let $a_i=b_i$ for some $i$. Since $\Delta$ is closed, all $m$-subsets of $\{a_1,\ldots,a_m\}\cup\{b_1,\ldots,b_m\}$ belong to $\Delta$. Therefore  $S([a_1\ldots a_m],[b_1\ldots b_m])$ reduces
to zero with respect to the $m$-subsets of $\{a_1,\ldots,a_m\}\cup\{b_1,\ldots,b_m\}$, and hence with respect to $\Gc$. Then by using Buchberger's criterion, it follows that $\Gc$ is a Gr{\"o}bner basis of $J_\Delta$.

\medskip

Assume that $\Gc$ is a Gr{\"o}bner basis for the ideal $J_\Delta$.  Let  $[a_1a_2\ldots a_m]$  with $a_1<a_2<\cdots <a_m$ and $[b_1b_2\ldots b_m]$ with
$b_1<b_2<\cdots <b_m$  belong to $\Gc$, and assume  that  $a_i=b_i$ for some $i$.

We will show that $\Delta$ is closed. The proof is by descending induction on $$k=|\{a_1,\ldots,a_m\}\cap\{b_1,\ldots,b_m\}|.$$

First assume that $k=m-1$. Then there exists an integer $\ell$ such that $a_1=b_1, \ldots,a_{\ell-1}=b_{\ell-1}$ and $a_\ell\neq b_\ell$. We may assume $b_\ell<a_\ell$. Then
\[
\{b_1<\ldots<b_m\}=\{a_1<a_2<\cdots <a_{\ell-1}<b_\ell<a_\ell<\cdots <a_{\ell'-1}<a_{\ell'+1}<\cdots <a_m\}
\]
for some $\ell'\geq \ell$.

In order to prove that  in this case $\Delta$ is closed we have to show that
\[
\{a_1,\ldots,a_m,b_\ell\}\setminus \{a_r\}\in \Delta
\]
for all $r$.

Since $a_i=b_i$ for some $i$ we have that either $\ell'<m$ or $1<\ell$. We first  assume that $\ell'< m$, and choose an integer $r$ with $\ell'<r\leq m$.

Then we  use the determinantal identity of Lemma~\ref{det} for $\{d_1<\cdots<d_{m-2}\}$ equal to
\[
\{a_1<\cdots <a_{\ell-1}<a_{\ell}<\cdots <a_{\ell'-1}< a_{\ell'+1}< \cdots <a_{r-1} <a_{r+1}< \cdots < a_m\}
\]
and $\{e_1<e_2<e_3\}=\{b_\ell< a_{\ell'}<a_r\}$,
  and obtain
\begin{eqnarray*}
&(-1)^{\ell'+1}&[1 \ldots m-1|a_1\ldots \hat{a}_{\ell'}\ldots a_m ][a_1 \ldots a_{\ell-1}b_\ell a_\ell\ldots  \hat{a}_r \ldots a_{m}]\\
&+(-1)^{r+1}&[1 \ldots m-1|a_1\ldots \hat{a}_{r} \ldots a_m][b_1\ldots b_m]\\
&+(-1)^{\ell}&[1\ldots m-1|a_1\ldots a_{\ell-1}b_\ell a_\ell\ldots \hat{a}_{\ell'} \ldots \hat{a}_{r}\ldots a_m ][a_1 \ldots a_m]=0.
\end{eqnarray*}

Since the last two terms are in $J_\Delta$ and $\Gc$ is a Gr\"obner basis for $J_\Delta$, the initial monomial  of the first term is divisible by the initial monomial of a minor in $\Gc$.

The initial monomial of the first term is
\begin{eqnarray*}
u=(x_{1a_1}&\cdots& x_{\ell'-1 a_{\ell^\prime-1}} x_{\ell^\prime a_{\ell^\prime+1}} \cdots x_{m-1 a_m})\\
&&\times (x_{1 a_1} \cdots x_{\ell-1 a_{\ell-1}} x_{\ell b_\ell} x_{\ell+1 a_{\ell}}    x_{\ell+2 a_{\ell+1}} \cdots x_{r a_{r-1}} x_{r+1 a_{r+1}}\cdots x_{m a_m}).
\end{eqnarray*}
It follows that  ${\rm in_<}[a_1 \ldots a_{\ell-1} b_\ell a_\ell \ldots  \hat{a}_{r} \ldots a_m]$
is the only initial monomial of a maximal minor of $X$ which divides the above monomial. Indeed, to find the initial monomial of a maximal minor
which divides $u$ we need to choose an increasing subsequence of $a_1< \cdots <a_{\ell-1}<b_\ell<a_\ell<a_{\ell+1}<\cdots< a_m $ with $m$ elements. Note that
for the first $\ell-1$ and the last $m-r$ elements we have a unique choice, namely $a_1<\cdots <a_{\ell-1}$ and, respectively, $a_{r+1}<\cdots <a_m.$ Hence we have to choose a subsequence with $r-\ell+1$ elements of $b_\ell<a_\ell<a_{\ell+1}<\cdots< a_r.$ Now we observe that $x_{ra_r}$ does not divide $u$, hence we  cannot keep $a_r$ in the above sequence. Therefore, the unique choice of the subsequence is $b_\ell<a_\ell<a_{\ell+1}<\cdots< a_{r-1}$.

Hence we deduce that
\[
[a_1 \ldots a_{\ell-1} b_\ell a_\ell \ldots  \hat{a}_{r} \ldots a_m]\in \Gc
\]
and so $\{a_1, \ldots, a_{\ell-1}, b_\ell, a_\ell, \ldots,  \hat{a}_{r}, \ldots, a_m\}$ is in $\Delta$ for all $r>\ell'$.

Next we assume that $1<\ell$. Then we deduce as above that $$\{a_1,\ldots,  \hat{a}_{r}, \ldots, a_{\ell-1}, b_\ell, a_\ell,  \ldots, a_m\}$$ is in $\Delta$ for $r<\ell$. More precisely, we use again Lemma~\ref{det}, but for $[c_1\ldots c_{m-1}]=[2\ldots m]$ and get the following identity:

\begin{eqnarray*}
&(-1)^{r}&[2 \ldots m|a_1\ldots \hat{a}_{r}\ldots a_m ][b_1 \ldots  b_{m}]\\
&+(-1)^{\ell-1}&[2 \ldots m|a_1\ldots \hat{a}_{r}\ldots b_\ell a_\ell\ldots \hat{a}_{\ell'} \ldots a_m][a_1\ldots a_m]\\
&+(-1)^{\ell'-1}&[2\ldots m|a_1\ldots  \hat{a}_{\ell'} \ldots  a_m ][a_1 \ldots \hat{a}_{r}\ldots b_\ell a_\ell\ldots a_m]=0.
\end{eqnarray*}

The last term in this identity belongs to $J_\Delta$, thus its initial monomial is divisible by the initial monomial of a minor in $\Gc.$ By using similar arguments as before, we get the claim.

Finally we show that for {\em all} $r$ we have $\{a_1,\ldots,a_m,b_\ell\}\setminus \{a_r\}\in \Delta$. To this end we may assume that $\ell'<m$  and choose $r=\ell'+1$, to obtain by the above arguments that $\{a_1,\ldots,  a_{\ell-1}, b_\ell, a_\ell,  \ldots,  \hat{a}_{\ell'+1}, \ldots,a_m\}$ is a facet  of $\Delta$. Comparing this facet with the facet $\{a_1,\ldots,  a_{\ell-1}, b_\ell, a_\ell,  \ldots,  \hat{a}_{\ell'}, \ldots,a_m\}$ of $\Delta$ it follows from the above considerations that $\{a_1,\ldots,  a_{\ell-1}, b_\ell, a_\ell,  \ldots,a_m\}\setminus \{a_r\}\in \Delta$ for all $r\leq  \ell'$.

\medskip
Assume now that
$
|\{a_1,\ldots,a_m\}\cap\{b_1,\ldots,b_m\}|=k<m-1.
$
Let $s$ be the number of integers $i$ such that $a_i=b_i$.  By our assumption, $s\geq 1$, and of course $s\leq k$.
Assume that  $a_1=b_1,\ldots,a_s=b_s$ and  $a_{s+1}<b_{s+1}$.
Then
\begin{eqnarray*}
&&\ini_<([s+1\ldots m| b_{s+1}\ldots b_m][a_1\ldots a_m]-
[s+1\ldots m| a_{s+1}\ldots a_m][b_1\ldots b_m])\\
&=&(x_{s+1 b_{s+1}}\cdots x_{m b_m})(x_{1a_1}\cdots x_{s-1 a_{s-1}} x_{s a_{s+1}} x_{s+1 a_s} x_{s+2 a_{s+2}}
\cdots x_{ma_m})=u,
\end{eqnarray*}
because the monomials bigger than $u$ in the expression whose initial monomial we compute  cancel.
Therefore
there exists a minor $[c_1\ldots c_m]$ in $\Gc$ with $c_1<c_2<\cdots <c_m$ such that ${\rm in_<}[c_1\ldots c_m]$ divides
the monomial $$(x_{s+1 b_{s+1}}\cdots x_{m b_m})(x_{1a_1}\cdots x_{s-1 a_{s-1}} x_{s a_{s+1}} x_{s+1 a_s} x_{s+2 a_{s+2}}
\cdots x_{ma_m}),$$
and we have
\[
c_1=a_1,\ldots,c_{s-1}=a_{s-1},c_s=a_{s+1},c_{s+1}=b_{s+1},\ {\rm and}\ c_\ell\in\{a_\ell,b_\ell\}\ {\rm for\ }\ell\geq s+2.
\]

First consider the case  $s=k$. Then $c_{m}$ is either $a_{m}$ or $b_{m}$, and we may assume that $c_{m}=a_{m}$.
Therefore $|\{c_1,\ldots,c_m\}\cap \{a_1,\ldots,a_m\}|>k$.
Applying the  inductive hypothesis for the facets $\{c_1,\ldots,c_m\}$ and $\{a_1,\ldots,a_m\}$ of $\Delta$, we conclude that all $m$-subsets of $$\{a_1,\ldots,a_m\}\cup\{c_{1},\ldots,c_m\}$$
belong to $\Delta$.

Note that
there exists some $c_i$ such that $c_i\not\in\{a_1,\ldots,a_m\}$, since $a_s\not\in\{c_1,\ldots,c_m\}$.  It follows that  $c_i=b_i$, and consequently $b_i\not\in \{a_1,\ldots,a_m\}$.
Moreover, since  $k<m-1$  there exist two integers $j_1$ and $j_2$ such that
$$a_{j_1},a_{j_2}\not\in\{b_1,\ldots,b_m\}.$$

Since $\{a_1,\ldots,\hat{a}_{j_1},\ldots,a_m,b_{i}\}$ and  $\{a_1,\ldots,\hat{a}_{j_2},\ldots,a_m,b_{i}\}$ are $m$-subsets of $$\{a_1,\ldots,a_m\}\cup\{c_{1},\ldots,c_m\},$$ these sets belong to $\Delta$.
Now applying the inductive hypothesis to the sets $\{b_1,\ldots, b_m\}$ and
$\{a_1,\ldots,\hat{a}_{j_1},\ldots,a_m,b_{i}\}$ which intersect in $k+1$ elements,  we will get all $m$-subsets of $$\{a_1,\ldots,\hat{a}_{j_1},\ldots,a_m,b_{i}\}\cup\{b_1,\ldots,b_m\}$$ in $\Delta$.
By the same argument we deduce that  all $m$-subsets of $$\{a_1,\ldots,\hat{a}_{j_2},\ldots,a_m,b_{i}\}\cup\{b_1,\ldots,b_m\}$$ belong to $\Delta$.

Now assume that $F$ is an arbitrary subset of $\{a_1,\ldots,a_m\}\cup\{b_1,\ldots,b_m\}$ such that $a_{j_1},a_{j_2}\in F$ and $b_{j}\not\in F$ for some $j$. By the above statements we have
$(F\setminus \{a_{j_1}\})\cup\{b_j\}$ and $(F\setminus \{a_{{j_2}}\})\cup\{b_j\}$ in $\Delta$. Then comparing these two facets we deduce that $F\in \Delta$, since their intersection has cardinality $m-1$.

\medskip

We remark that in  the more general  case that $a_{\ell_1}=b_{\ell_1},\ldots,a_{\ell_s}=b_{\ell_s}$, the proof is similar. We just consider the minor
\[
[1\ldots \hat{\ell}_1 \ldots \hat{\ell}_s\ldots m| a_1\ldots \hat{a}_{\ell_1}\ldots \hat{a}_{\ell_s}\ldots a_m]
\]
instead of $[s+1\ldots m| a_{s+1}\ldots a_m]$ and the minor
\[
[1\ldots \hat{\ell}_1 \ldots \hat{\ell}_s\ldots m| b_1\ldots \hat{b}_{\ell_1}\ldots \hat{b}_{\ell_s}\ldots b_m]
\]
instead of $[s+1\ldots m| b_{s+1}\ldots b_m]$ to get the desired minors in $\Gc$.

Therefore, the assertion of the theorem is proved if $s=k$.
\medskip

Now assume that $s<k$, and for every two sets in $\Delta$ with $k$ common elements
which at least $s+1$ of them have the same position in both sets, the result holds. Let $a_{\ell_1}=b_{t_1},\ldots, a_{\ell_{k-s}}=b_{t_{k-s}}$ for some integers $\ell_1<\cdots<\ell_{k-s}$ and $t_1<\cdots<t_{k-s}$, where $t_r\neq \ell_r$ for $r=1,\ldots,k-s$.
Assume that $$\{a_{\ell_{\sigma_1}},\ldots,a_{\ell_{\sigma_p}}\}\subset \{c_{s+2},\ldots,c_m\},\; \text{and}\; \{a_{\ell_{\tau_1}},\ldots,a_{\ell_{\tau_q}}\}\not\subset \{c_{s+2},\ldots,c_m\},$$ for $\{\sigma_1,\ldots,\sigma_p,\tau_1,\ldots,\tau_q\}=\{\ell_1,\ldots,\ell_{k-s}\}$.
First assume that $p=k-s$. Note that there exists some index $j$ with
$j\not\in \{1,\ldots,s+1,\ell_1,\ldots,\ell_{k-s}\}$, since $k<m-1$. If $c_j=a_j$ for some $j\not\in \{1,\ldots,s+1,\ell_1,\ldots,\ell_{k-s}\}$,
then $|\{a_1,\ldots,a_m\}\cap\{c_1,\ldots,c_m\}|>k$, and by the inductive hypothesis we will get all $m$-subsets of $\{a_1,\ldots,a_m\}\cup\{c_{1},\ldots,c_m\}$ in $\Delta$. Otherwise $|\{b_1,\ldots,b_m\}\cap\{c_1,\ldots,c_m\}|>k$, and by the inductive hypothesis all $m$-subsets of $\{b_1,\ldots,b_m\}\cup\{c_{1},\ldots,c_m\}$ belong to $\Delta$.
In both cases applying the same argument as in the case $s=k$, we deduce that all desired $m$-subsets  are in $\Delta$.

Now assume that $p<k-s$.
We claim that
$$c_{\ell_{r}}=b_{\ell_{r}} \quad \text{for}\quad r=\tau_1,\ldots,\tau_q,$$
in particular, $\{b_{\ell_{\tau_1}},\ldots,b_{\ell_{\tau_q}}\}\subset \{c_1,\ldots,c_m\}.$

Indeed, suppose that
$a_{\ell_{r}}\not\in\{c_{s+2},\ldots,c_m\}$. Therefore $c_{\ell_{r}}=b_{\ell_r}$ and $c_{t_{r}}=a_{t_{r}}$.
\medskip

Since $a_{s+1}<b_{s+1}<\cdots<b_m$, we have $\ell_r>s+1$ for all $r$. Therefore
$$c_1=b_1,\ldots,c_{s-1}=b_{s-1},c_{s+1}=b_{s+1},c_{\ell_{\tau_1}}=b_{\ell_{\tau_1}},\ldots,c_{\ell_{\tau_q}}=b_{\ell_{\tau_q}}, $$
$$c_{\ell_{\sigma_1}}=a_{\ell_{\sigma_1}}=b_{t_{\sigma_1}},\ldots,c_{\ell_{\sigma_p}}=a_{\ell_{\sigma_p}}=b_{t_{\sigma_p}},$$
which shows that $\{c_1,\ldots,c_m\}$ and $\{b_1,\ldots,b_m\}$ have at least $k$ common elements
and $s+q\geq s+1$ of them have the same position in  both sets.
Now applying the result of the first case to  these two sets, we deduce that all $m$-subsets of $$\{b_1,\ldots,b_m\}\cup\{c_{1},\ldots,c_m\}$$
are in $\Delta$. Now the same argument as in case $k=s$, for  $\{b_1,\ldots,b_m\}\cup\{c_{1},\ldots,c_m\}$ instead of
$\{a_1,\ldots,a_m\}\cup\{c_{1},\ldots,c_m\}$,
implies that  all desired $m$-subsets belong to $\Delta$.
\end{proof}

For determinantal facet ideals of closed simplicial complexes we may compute important numerical invariants.

\begin{Corollary}
\label{cm}
Let $\Delta$ be a closed simplicial complex of dimension $(m-1)$ and let $\Delta=\Delta_1\union\Delta_2\union\cdots\union \Delta_r$ its clique decomposition. For $1\leq \ell\leq r$ let $n_\ell$ be the number of vertices of $\Delta_\ell.$ Then:
\begin{itemize}
	\item[(a)]
	$\height J_\Delta=\sum_{\ell=1}^r\height J_{\Delta_\ell}=\sum_{\ell=1}^r n_\ell-(m-1)r.$
	\item [(b)] $J_\Delta$ is Cohen-Macaulay.
	\item [(c)] The Hilbert series of $S/J_\Delta$ has the form
	\[
	H_{S/J_\Delta}(t)=\frac{\prod_{\ell=1}^r Q_{\ell}(t)}{(1-t)^{mn-\sum_{\ell=1}^r n_\ell +(m-1)r}},
	\]
	where
	\[
	Q_\ell(t)=[\det(\sum_k{m-i \choose k}{n_\ell-j\choose k})_{1\leq i,j\leq m-1}]/ t^{m-1\choose 2}
	\]
	 for $1\leq \ell\leq r.$
	\item [(d)] The multiplicity of $S/J_\Delta$ is
	\[
	e(S/J_\Delta)=\prod_{\ell=1}^r {n_\ell \choose m-1}.
	\]
\end{itemize}
\end{Corollary}

\begin{proof}
It follows from  characterization  (2) of closed simplicial complexes that  the initial ideals  ${\ini}_<(J_{\Delta_\ell})$ are monomial ideals in
disjoint sets of variables, therefore the first equality in (a) is obvious. The second equality follows from the known formula of the height of
determinantal ideals, see for instance \cite[Theorem 6.35]{EH}.

By \cite[Corollary 3.3.5]{HHBook}, $S/J_{\Delta}$ and $S/\ini_<(J_\Delta)$ have the same Hilbert series. By \cite[Corollary 1]{CH} or
\cite[Theorem 6.9]{BC} and \cite[Theorem 3.5]{HT}, we know formulas for the Hilbert series and multiplicity for determinantal rings defined by maximal minors. Therefore, (c) and (d) follows once we observe that, by characterization (2) of closed simplicial complexes, we have
\begin{equation}\label{tensor}
S/\ini_<(J_\Delta)\cong \bigotimes_{i=1}^r S_i/\ini_<(J_{\Delta_i})
\end{equation}
where $S_i$ are polynomial rings in disjoint sets of variables whose union is the set of all the variables of $X.$
 By using again relation~(\ref{tensor}), since all factors in the right hand side are Cohen-Macaulay (see \cite{CGG} and \cite{St90}),  it follows that  ${\ini}_<(J_\Delta)={\ini}_<(J_{\Delta_1})+\cdots+{\ini}_<(J_{\Delta_r})$ is
also Cohen--Macaulay. This implies that $J_\Delta$ is Cohen--Macaulay, see for example \cite[Corollary 3.3.5]{HHBook}.
\end{proof}

\begin{Corollary}
\label{gorenstein}
Suppose that  $\Delta$ is closed with clique decomposition $\Delta=\Delta_1\cup\ldots\cup \Delta_r$.  Then the $K$-algebra
\[
A=K[\{[a_1\ldots a_m]:\ \{a_1,\ldots, a_m\}\in\Delta\}]
\]
is Gorenstein of dimension $r+\sum_{i=1}^rm(n_i-m)$, where $n_i$ is the cardinality of the vertex set of $\Delta_i$.
\end{Corollary}

\begin{proof}  We first observe that
\begin{eqnarray*}
B:= K[\{{\ini}_<[a_1\ldots a_m]\: \{a_1,\ldots, a_m\}\in\Delta\}]\iso \bigotimes_{i=1}^r
K[\{{\rm in}_<[a_1\ldots a_m]\: \{a_1,\ldots, a_m\}\in\Delta_i\}].
\end{eqnarray*}
We use the Sagbi basis criterion (see \cite[Theorem 6.43]{EH}) which asserts that the minors $[a_1\ldots a_m]$ with $\{a_1,\ldots, a_m\}\in\Delta$ form a Sagbi basis of $A$, that is, the monomials $[a_1\ldots a_m]$ with $\{a_1,\ldots, a_m\}\in\Delta$ generate the initial algebra $\ini_<(A)$,  if a generating set of binomial relations of the algebra $B$ can be lifted. It follows from the tensor presentation of $B$ that a set of binomial relations of $B$ is obtained as the union of the binomial relations of each of the algebras $K[\{{\rm in}_<[a_1\ldots a_m]\: \{a_1,\ldots, a_m\}\in\Delta_i\}]$. For these algebras it is known  that they admit a set of liftable relations. Thus it follows that $B=\ini_<(A)$.

Next we observe  that for each $i$, the $K$-algebra $K[\{{\rm in}_<[a_1\ldots a_m]:\ \{a_1,\ldots, a_m\}\in\Delta_i\}]$
is the Hibi ring associated to the distributive lattice $\Lc_i$ of all maximal $m$-minors
$[a_1\ldots a_m]$ with $\{a_1,\ldots, a_m\}\in\Delta_i$ whose partial order is given by
\[
[a_1\ldots a_m]\leq [b_1,\ldots,b_m] \quad \Leftrightarrow \quad a_i\leq b_i \quad \text{for} \quad i=1,\ldots,m.
\]
The distributive lattice $\Lc_i$ is graded, which by a theorem of Hibi \cite{Hi} implies that  $$K[\{{\rm in}_<[a_1\ldots a_m]:\ \{a_1,\ldots, a_m\}\in\Delta_i\}]$$ is Gorenstein. It follows that  $A$ is Gorenstein; see \cite[Theorem 3.16]{BC}.

Finally we notice that
\begin{eqnarray*}
\dim A=\dim \ini_<(A)&=&\sum_{i=1}^r\dim K[\{{\rm in}_<[a_1\ldots a_m]\: \{a_1,\ldots, a_m\}\in\Delta_i\}]\\
&=&\sum_{i=1}^r \dim K[\{[a_1\ldots a_m]\: \{a_1,\ldots, a_m\}\in\Delta_i\}].
\end{eqnarray*}
The desired formula for the dimension of $A$ follows, because  $K[\{[a_1\ldots a_m]\: \{a_1,\ldots, a_m\}\in\Delta_i\}]$ is the algebra  of all maximal minors of an $m\times n_i$-matrix of indeterminates, and hence its dimension is  $m(n_i-m)+1$; see, for example, \cite[Theorem 6.45]{EH}.
\end{proof}

\section{Primality of  determinantal facet ideals}

In this and the following section we want to discuss when a determinantal facet ideal is a prime ideal. In general $J_\Delta$ need not  be a prime ideal even if $\Delta$ is closed.  For example, if $\Delta$ is the simplicial complex with facets  $\mathcal{F}(\Delta)=\{\{1,2,3\}, \{2,3,4\}\}$ or  $\mathcal{F}(\Delta)=\{\{1,2,3\},\{2,3,6\},\{3,4,5\}\}$, then $J_\Delta$ is not a prime ideal. Indeed, in the first case,
$\height J_\Delta=\height \ini_<(J_\Delta)=2$ since $\ini_<(J_\Delta)$ is generated by a regular sequence of length $2$ and $P=(x_2y_3-x_3y_2,x_2z_3-x_3z_2,y_2z_3-y_3z_2)$ is a prime ideal of height $2$ which obviously strictly contains $J_\Delta.$ Here we denoted the variables of the first row of $X$ by $x,$ of the second row by $y,$ and of the third row by $z$ together with appropriate indices. In the second case, we get $\height J_\Delta=\height \ini_<(J_\Delta)=3$ and $J_\Delta\subsetneq (x_3,y_3,z_3),$ hence clearly $J_\Delta$ is not prime. Even in this rather simple examples we  see  that the primary decomposition of determinantal facet ideals looks much more complicated than  that for binomial edge ideals.

 The main result of this section, Theorem~\ref{viviana}, explains why $J_\Delta$ is not a prime ideal in the above examples.

The proofs of primality that follow depend on localization with respect to nonzero divisors. This technique allows to use induction arguments. Indeed,  suppose we want to show that $J\subset S$ is a prime ideal. Then we are trying to find an element $f\in S$ which is regular modulo $J$. This implies that the natural map $S/J\to (S/J)_f$ is injective. Now, if we can find a prime ideal $L\subset S$ such that $(S/L)_f\iso (S/J)_f$, we conclude that $(S/J)_f$, and consequently $S/J$, is a domain which implies that $J$ is a prime ideal. This procedure often allows us to use inductive arguments, as in many cases  $L$ is of a simpler structure.

The next lemma  helps us to understand the effect of localization when we are dealing with ideals generated by minors of a matrix.

\begin{Lemma}
\label{rules}
Let $K$ be a field,   $X$ be an $m\times n$-matrix of indeterminates and $I\subset S=K[X]$ an ideal generated by a set $\Gc$ of minors. Furthermore,
let $x_{ij}$ be an entry of $X$. We assume that for each minor $[a_1\ldots a_t|b_1\ldots b_t]\in \Gc, t\geq 1,$ there exists $\ell $ such that
$a_\ell=i$ so that every minor of $\Gc$ involves the $i$th row.

Then $(S/I)_{x_{ij}}\iso (S/J)_{x_{ij}}$ where $J$ is generated by the minors  $[a_1\ldots a_t|b_1\ldots b_t]\in \Gc$ with $b_\ell\neq j$ for all
$\ell\in \{1,\ldots,t\},$ and the minors $[a_1\ldots\hat{a}_\ell\ldots a_t|b_1\ldots\hat{b}_k\ldots b_t]$ where $[a_1\ldots a_t|b_1\ldots b_t]\in \Gc$ and $a_\ell=i$ and
$b_k=j$.
\end{Lemma}

\begin{proof} For simplicity we may assume that $i=1$ and $j=1$. We apply the  automorphism $\varphi\: S_{x_{11}}\to S_{x_{11}}$ with
\[
x_{ij}\mapsto x^\prime_{ij}=\left\{
\begin{array}{ll}
	x_{ij}+x_{i1}x_{11}^{-1}x_{1j}, & \text{if } i\neq 1 \text{ and } j\neq 1,\\
	x_{ij}, & \text{if } i=1 \text{ or } j=1.
\end{array}
\right.
\]
Let $I'\subset S_{x_{11}}$ be the ideal which is the image of $IS_{x_{11}}$ under the automorphism  $\varphi$. Then $(S/I)_{x_{11}}\iso S_{x_{11}}/I'$. The ideal $I'$ is generated in $S_{x_{11}}$ by the elements $\varphi(\mu_M)$ where $\mu_M\in \Gc$.  Note that if $\mu_M=[a_1\ldots a_t|b_1\ldots b_t]$, then $\varphi(\mu_M)=\det(x_{a_ib_j}')_{i,j=1,\ldots,t}$.

In the following  we may assume that $a_1<a_2<\cdots <a_t$ and $b_1<b_2<\cdots<b_t$ for $\mu_M=[a_1\ldots a_t|b_1\ldots b_t]\in \Gc$. Then our assumption implies that $a_1=1$. Let us first consider the case that $b_1\neq 1$. Then $\varphi(\mu_M)$ is the determinant of the matrix
\[
\left(
  \begin{array}{cccc}
    x_{1b_1} & x_{1b_2} & \cdots & x_{1b_t} \\
     x_{a_2b_1}+x_{a_21}x_{11}^{-1}x_{1 b_1}&  x_{a_2b_2}+x_{a_21}x_{11}^{-1}x_{1 b_2} & \cdots &  x_{a_2b_t}+x_{a_21}x_{11}^{-1}x_{1b_t}\\
    \vdots & \vdots & \cdots & \vdots  \\
     x_{a_tb_1}+x_{a_t1}x_{11}^{-1}x_{1 b_1}&  x_{a_tb_2}+x_{a_t1}x_{11}^{-1}x_{1 b_2} & \cdots &  x_{a_tb_t}+x_{a_t1}x_{11}^{-1}x_{1 b_t}\\
     \end{array}
\right)
\]
By subtracting  suitable multiples of the first row from the other rows we see that $$\varphi(\mu_M)=\det(x_{a_ib_j})_{i,j=1,\ldots,t}=\mu_M.$$

In the case that $b_1=1$, the element $\varphi(\mu_M)$ is the determinant of the matrix
\[
\left(
  \begin{array}{cccc}
    x_{11} & x_{1b_2} & \cdots & x_{1b_t} \\
      x_{a_21}&  x_{a_2b_2}+x_{a_21}x_{11}^{-1}x_{1 b_2} & \cdots &  x_{a_2b_t}+x_{a_21}x_{11}^{-1}x_{1b_t}\\
    \vdots & \vdots & \cdots & \vdots  \\
     x_{a_t1}&  x_{a_tb_2}+x_{a_t1}x_{11}^{-1}x_{1 b_2} & \cdots &  x_{a_tb_t}+x_{a_t1}x_{11}^{-1}x_{1 b_t}\\
     \end{array}
\right)
\]
Applying suitable row operations we obtain the matrix
\[
\left(
  \begin{array}{cccc}
    1 & x_{11}^{-1}x_{1b_2} & \cdots & x_{11}^{-1}x_{1b_t} \\
     0&  x_{a_2b_2}& \cdots &  x_{a_2b_t}\\
    \vdots & \vdots & \cdots & \vdots  \\
     0&  x_{a_tb_2} & \cdots &  x_{a_tb_t}\\
     \end{array}
\right)
\]
It follows that $\varphi(\mu_M)=\det(x_{a_ib_j})_{i,j=2,\ldots,t}$. These calculations show that $I'=JS_{x_{11}}$, as desired.
\end{proof}

Now we are ready to prove

\begin{Theorem}
\label{viviana}
Let $m\leq n$, let $\Delta$ be a pure $(m-1)$-dimensional closed simplicial complex on the vertex set $[n]$ and let $\Delta=\Delta_1\union \ldots\union \Delta_r$
be the clique decomposition  of $\Delta$.
If $J_{\Delta}$ is a prime ideal, then for all $2\leq t\leq \min(m,r)$ and for any pairwise distinct cliques $\Delta_{i_1},\ldots,\Delta_{i_t}$ we
have
\[
|V(\Delta_{i_1})\cap\cdots \cap V(\Delta_{i_t})|\leq m-t.
\]
\end{Theorem}

\begin{proof} We make induction on $m.$ The initial step, $m=2$, is already known \cite{HHH}.

Let us make the inductive step. We first consider $t<m.$ Let us assume that there exist $\Delta_{i_1},\ldots,\Delta_{i_t}$ such that
$|V(\Delta_{i_1})\cap\cdots \cap V(\Delta_{i_t})|> m-t.$ Without loss of generality we may assume  that $V(\Delta_1)\cap\cdots \cap
V(\Delta_t)=\{a_1,a_2,\ldots,a_\ell\}$ with $\ell\geq m-t+1$ and
$1\leq a=a_1 <\cdots < a_\ell\leq n$. We may further assume that there exists $s\geq t$ such that $a\in V(\Delta_i)$ for $1\leq i\leq s$ and $a\notin V(\Delta_i)$ for
$s+1\leq i\leq r.$ Since $J_\Delta$ is prime, it follows that $x_{ma}$ is regular on $J_\Delta$ and $J_\Delta S_{x_{ma}}$ is also a prime ideal in the
localization $S_{x_{ma}}$ of $S$. Thus $(S/J_\Delta)_{x_{ma}}$ is a domain. By Lemma~\ref{rules}, it follows that $(S/J_\Delta)_{x_{ma}}\iso
(S/L)_{x_{ma}},$ where $L=L_1+\sum_{i=s+1}^rJ_{\Delta_i},$ and $L_1$ is the determinantal facet ideal of the closed $(m-2)$-dimensional simplicial
complex $\Delta^\prime$ with the clique decomposition $\Delta^\prime=\Delta_1^\prime\union\cdots\union \Delta_s^\prime,$ where
$\Delta_i^\prime=\langle F\setminus \{a\}: F\in\mathcal{F}(\Delta_i), a\in F\rangle$ for $1\leq i\leq s.$ As $\Delta^\prime_1,\ldots,\Delta^\prime_t$
intersect in $\ell-1\geq m-t$ vertices, by induction, it follows that $L_1$ is not a prime ideal which will imply, as we are going to show, that
$L$ is not a prime ideal. But this is a contradiction, since $(S/L)_{x_{ma}}$ must be a domain.

Since $L_1$ is not prime, there exist  polynomials $f,g$ in $S$ such that $fg\in L_1$ and $f,g\notin L_1.$ We claim that $f,g\notin L.$ Let us assume, for instance, that $f\in L.$ Then we may write $f=\sum_{G}h_G\gamma_G+ \sum_{F}h_F\mu_F$ for some
polynomials $h_G,h_F\in S$ where the first sum is taken over all  $G\in \bigcup_{i=1}^s {\mathcal{F}}(\Delta_i^\prime)$, and the second one over all
$F\in \bigcup_{i=s+1}^r {\mathcal{F}}(\Delta_i)$. Then, by mapping  the indeterminates $x_{mj}$ to zero  for all $j\neq a$ and $x_{ma}$ to $1,$ we get $f=\sum_{G}h_G^\prime \gamma_G$ for some polynomials $h_G^\prime \in S,$ thus $f\in L_1,$ a contradiction. Therefore, $L$ is not a prime ideal.

It remains to consider the case  $t=m.$ We may assume that $|V(\Delta_1)\cap\cdots \cap V(\Delta_m)|\geq 1.$ Let $a\in V(\Delta_1)\cap\cdots \cap V(\Delta_m)$. It is clear that $J_\Delta\subset (J_{\Delta^\prime},x_{1a},\ldots,x_{ma})$ where $\Delta^\prime=\{F\in \Delta: a\notin F\}.$ Since $\Delta$ is closed, it follows that $\Delta^\prime$ is closed as well and, moreover, by using Corollary~\ref{cm},
\[
\height J_{\Delta^\prime}=\sum_{i=1}^m ((n_i-1)-m+1)+\sum_{i=m+1}^r(n_i-m+1)=\height J_{\Delta}-m.
\]
Since $x_{1a},\ldots,x_{ma}$ is obviously a regular sequence on $S/J_{\Delta^\prime},$ we have
\[
\height(J_{\Delta^\prime},x_{1a},\ldots,x_{ma})=\height J_{\Delta^\prime}+m=\height J_{\Delta}.
\]
Let $P$ be a minimal prime of $(J_{\Delta^\prime},x_{1a},\ldots,x_{ma})$
of height equal to $\height(J_\Delta).$ Since $J_\Delta$ and $P$ are prime ideals of the same height, we must have $J_\Delta=P.$ But  $P$ contains
the indeterminates $x_{1a},\ldots,x_{ma},$ which do not belong to $J_{\Delta}.$ Therefore, we have got a contradiction.
\end{proof}

The proofs of primality that follow depend on localization with respect to nonzero divisors. The next result tells us that in our situation all variables are nonzero divisors.

\begin{Lemma}
\label{lasthope}
Let  $\Delta$ be  closed $(m-1)$-dimensional simplicial complex with the property that any $m$ pairwise distinct cliques of $\Delta$ have an empty intersection.  Then each of the variables $x_{ij}$ is regular modulo $J_\Delta$.
\end{Lemma}

\begin{proof} We may assume from the beginning that the field $K$ is infinite since neither the hypothesis nor the conclusion of the lemma is
affected by tensoring with a field extension of $K$. In order to show that $x_{ij}$ is regular modulo $J_\Delta$ we consider the ideal $$I=(J_\Delta,x_{1j}, \ldots,x_{mj}).$$ Let $\Delta'$ be the simplicial complex whose facets are those of $\Delta$ which do not contain $j$. Observe that $\Delta'$ is again closed, and that $I=(J_{\Delta'}, x_{1j}, \ldots,x_{mj})$. We use the formula in Corollary~\ref{cm} to compare the height of $I$ with that of $J_\Delta$. If $\Delta=\Delta_1\union\cdots \union \Delta_r$ is the clique decomposition of $\Delta$ with $n_i=|\Delta_i|$, then $\height J_\Delta=\sum_{i=1}^r(n_i-m+1)$.

We may assume that $\Delta_i$ contains the vertex $j$ for $i=1,\ldots,s$.  Our assumptions implies that $s\leq m-1$. Note that the clique decomposition of $\Delta'=\Delta_1'\union\cdots \union \Delta_r'$ where the facets of each $\Delta_i'$ are those facets of $\Delta_i$ which do not contain $j$. It follows that $|\Delta_i'|=|\Delta_i|-1=n_i-1$ for $i=1,\ldots,s$ and $\Delta_i'=\Delta_i$ for $i>s$. Hence we get
\begin{eqnarray*}
\height I&= & \height J_{\Delta'} +m = \sum_{i=1}^s(n_i-1-m+1)+ \sum_{i=s+1}^r(n_i-m+1)+m\\
&=&\height J_\Delta-s+m>\height J_\Delta.
\end{eqnarray*}
Our considerations show that $I/J_\Delta\subset S/J_\Delta$ has positive height. Since  $S/J_\Delta$ is Cohen--Macaulay  and $K$ is infinite, it follows that a  generic linear combination $a_1x_{1j}+a_2x_{2j}+\cdots +a_mx_{mj}$ of the variables $x_{1j}, \ldots,x_{mj}$ (whose residue classes  generate $I/J_\Delta$) is regular modulo $J_\Delta$. Since the above linear combination is generic, we may assume that $a_{i}=1$.

Now  we consider the linear automorphism $\varphi\: S \to S$ with $\varphi(x_{ik})= a_1x_{1k}+a_2x_{2k}+\cdots +a_mx_{mk}$ for $k=1,\ldots,n$ and $\varphi(x_{\ell k})=x_{\ell k}$ for $\ell\neq i$ and all $k$. Let $X'$ be the matrix whose entries are the elements $\varphi(x_{\ell k})$ for $\ell=1,\ldots,m$ and $k=1,\ldots, n$. Then $X'$ is obtained from $X$ by elementary row operations. It follows that $\varphi(J_\Delta)=J_\Delta$.

By  our choice of $\varphi$ we have that $y_{ij}= \varphi(x_{ij})$ is regular modulo $J_\Delta$. Since $J_\Delta=\varphi(J_\Delta)$ it follows that $x_{ij}= \varphi^{-1}(y_{ij})$ is regular modulo $\varphi^{-1}(J_\Delta)= \varphi^{-1}(\varphi(J_\Delta))=J_\Delta$, as desired.
\end{proof}

We do not know whether for a closed simplicial complex $\Delta$ the necessary condition for $J_\Delta$ to be a prime ideal given  in Theorem~\ref{viviana}  is also sufficient. For the moment we can only present a partial converse of this result.

\begin{Proposition}
\label{fatemehinverse}
Let $\Delta$ be a simplicial complex with clique decomposition $\Delta=\Delta_1\union\Delta_2\union \cdots \union \Delta_r$. Assume that all cliques are simplices of dimension $m-1$ and that
\begin{itemize}
\item[(1)] $|V(\Delta_r)\sect \cdots \sect V(\Delta_{r-s+1})|\leq m-s$ for $s=2,\ldots,r$;
\item[(2)] $V(\Delta_{i_1})\sect \cdots \sect V(\Delta_{i_s})\subset V(\Delta_r)\sect \cdots \sect V(\Delta_{r-s+1})$ for all subsets $\{i_1,\ldots,i_s\}\subset [r]$ of cardinality $s$ with  $2\leq s\leq r$.
\end{itemize}
Then $J_\Delta$ is a prime ideal.
\end{Proposition}

\begin{proof}
We make induction on $m$.
The initial step, $m=2$, is already known \cite{HHH}. Assume that $|V(\Delta_1)\cap\cdots\cap V(\Delta_r)|=k$.
We consider the following labeling on the vertices of $\Delta$ such that
\begin{eqnarray*}
V(\Delta_\ell)=\{
a_{\ell 1}<\cdots< a_{\ell, m-k-\ell+1}<b_{1}<\cdots<b_k<c_{\ell 1}<\cdots<c_{\ell, \ell-1}\}
\end{eqnarray*}
for all $\ell=1,\ldots,r$, where the numbers $a_{ij}$ are pairwise distinct, and for each $s=2,\ldots,r$ we choose $c_{ij}$ such that
$$c_{r j}=c_{r-1, j}=\cdots=c_{r-s+1, j},$$ for $j=1,\ldots,|V(\Delta_r)\sect \cdots \sect V(\Delta_{r-s+1})|-k$.

Then with respect to this labeling, $\Delta$ is closed and so $x_{mb_1}$ is a regular element modulo $J_\Delta$ by Lemma~\ref{lasthope}.
It follows from Lemma~\ref{rules} that $(S/J_\Delta)_{x_{mb_1}}\iso (S/L)_{x_{mb_1}}$ where $L=\sum_{i=1}^r L_i$. Since $x_{m b_1}$ is regular modulo $J_\Delta$,  $J_\Delta$ is a prime ideal if and only if $L_{x_{m b_1}}$ is a prime ideal.
Here $L_i$ is generated by the minor
$$[1\ldots m-1|a_{i 1}\ldots  a_{i, m-k-i+1}b_{2} \ldots b_k c_{i1}\ldots c_{i,i-1}].$$

Let $\Delta'$ be the $(m-2)$-simplicial complex with the clique decomposition $\Delta'_1\cup\cdots\cup\Delta'_t$, where $\Delta'_i=\Delta_i\setminus\{b_1\}$.
Note that  conditions $(1)$ and $(2)$ hold for $\Delta'$. Therefore $L=J_{\Delta'}$ is a prime ideal by the inductive hypothesis.
\end{proof}

\begin{Example}
\label{necessary}
\rm{
Let $\Delta=\Delta_1\cup\cdots\cup \Delta_r$ with the assumption of Theorem~\ref{fatemehinverse}. Then we can describe the vertices of each
$\Delta_i$ in a nice way as the $i^{\rm th}$ row of a simple matrix. As an example let $m=6, r=4$, $|V(\Delta_4)\cap V(\Delta_3)|=3$,
$|\bigcap_{i=2}^4 V(\Delta_i)|=3$ and $|\bigcap_{i=1}^4 V(\Delta_i)|=2$. Then by the proof of the above theorem  we get
\[
\left(
  \begin{array}{cccccccccc}
    1 & 2 & 3 & 4 & b_1 & b_2 \\
    5 & 6 & 7 & b_1 & b_2 & c_{1} \\
    8 & 9 & b_1 & b_2 & c_1 & c_2 \\
    10 & b_1 & b_2 & c_1 & c_{3} & c_{4} \\
     \end{array}
\right)
\]
which describes the labels of the $\Delta_1,\ldots,\Delta_4$.
}
\end{Example}

\begin{Example}
{\rm
Let $\mathcal{F}(\Delta)=\{\{1,2,3\},\{1,4,5\},\{3,5,6\},\{2,4,6\}\}$. Then one may check with \textsc{Singular} \cite{GPS} that $J_\Delta$ is not a prime ideal. However, the intersection condition of Theorem~\ref{viviana} holds for $\Delta$. Thus for a converse of Theorem~\ref{viviana} one should require that $\Delta$ is a closed simplicial complex.

This is also an example of a determinantal facet ideal whose initial ideal with respect to the lexicographic order is not squarefree though $J_\Delta$ is a radical ideal.
}
\end{Example}

\section{Special  classes  of prime  determinantal facet ideals}

Let $\Delta$ a  pure  simplicial complex of dimension $m-1\geq 2$ and $\Delta=\Delta_1\cup\ldots\cup\Delta_r$ its clique decomposition. In this section we pose the following intersection properties on the cliques of $\Delta$:
\begin{enumerate}
\item[(i)] $|V(\Delta_i)\sect V(\Delta_j)|\leq 1$ for all $i< j$;
\item[(ii)] $V(\Delta_i)\cap V(\Delta_j)\cap V(\Delta_k)=\emptyset$ for all $i<j<k$.
\end{enumerate}
Theorem~\ref{viviana} implies that for $m=3$ the conditions (i) and (ii) are satisfied whenever $J_\Delta$ is a prime ideal, and that for any $m\geq 3$ these two conditions imply the intersection conditions formulated in Theorem~\ref{viviana}.

In this section we will show  that whenever  $\Delta$ is closed, the conditions  (i) and (ii) imply primality of $J_\Delta$ under some additional assumptions depending on a graph which we are going to define now.

For the simplicial complex with the properties (i) and (ii), we let $G_\Delta$ be the simple graph  with vertex set  $V(G_\Delta)=\{v_1,\ldots,v_r\}$ and edge set
\[
E(G_\Delta)=\{\{v_i,v_j\}\:\;  V(\Delta_i)\sect V(\Delta_j)\neq \emptyset\}.
\]
In the following the phrase ``$\Delta$ is a simplicial complex with graph $G_\Delta$" will always imply that $\Delta$ satisfies the conditions (i) and (ii) (because otherwise $G_\Delta$ is not defined).

At present we are able to prove primality of $J_\Delta$ for certain classes of simplicial complexes $\Delta$ only under the additional assumption that these complexes are closed. The next lemma provides a necessary  condition for a simplicial complex to be closed.

\begin{Lemma}
\label{order}
Let $\Delta$ be a closed simplicial complex with graph $G_\Delta$. Then each vertex $v_i$ of $G_\Delta$ has order at most $\min\{|V(\Delta_i)|,2\dim(\Delta)\}$.
\end{Lemma}

\begin{proof}
We say that a vertex $\ell\in \Delta_i$ takes the position $s$ if there is an $(m-1)$-dimensional face $\{a_1<a_2<\cdots < a_m\}$ of $\Delta_i$ such that $\ell=a_s$. In the clique $\Delta_i$ there are exactly $\min\{|V(\Delta_i)|,2\dim(\Delta)\}$ vertices which do not take all $m$ positions. On the other hand the assumption (ii) implies that each of these vertices can intersect with at most one clique $\Delta_{j}$, (where $v_j$ is a neighbor of $v_i$) which completes the proof.
\end{proof}

Now we are ready to consider  primality of $J_\Delta$ for special classes of simplicial complexes.

\begin{Theorem}
\label{tree}
Let $\Delta$ be simplicial complex such that $G_\Delta$ is a tree. Then
\begin{enumerate}
\item[(a)] $J_\Delta$ is a prime ideal, if $\Delta$ is closed;

\item[(b)] $\Delta$ is closed, if and only if each vertex of $G_\Delta$ has order at most $\min\{|V(\Delta_i)|,2\dim(\Delta)\}$.
\end{enumerate}
\end{Theorem}

Let $\{i_1<\ldots <i_s\}\subset [m]$ and  $\{j_1<\ldots <j_t\}\subset [n]$. We denote by $X^{j_1j_2\ldots j_t}_{i_1\ldots i_s}$  the submatrix of $X$ with rows $i_1,\ldots,i_s$ and columns $j_1,\ldots, j_t$.  Observe that  Lemma~\ref{rules} implies the well-known fact that if $I$ is generated by all $m$-minors of $X^{j_1\ldots j_t}_{1\ldots m}$, then $I_{x_{ij_k}}$ is generated by all  $(m-1)$-minors of the matrix  $X^{j_1\ldots \hat{j}_k\ldots j_t}_{1\ldots \hat{i}\ldots m}$.

\begin{proof}[Proof of Theorem~\ref{tree}]
(a) We may assume that $\Delta$ is a connected $(m-1)$-dimensional simplicial complex and that $\Delta=\Delta_1\union \Delta_2\union\cdots \union \Delta_r$ is the clique decomposition of $\Delta$. The proof is by induction on the number of cliques of $\Delta$ (which is the number of vertices of $G_\Delta$). We may assume that $v_1$ is a vertex of degree one in $G_\Delta$ and that $v_2$ is its neighbor.  Then $\Delta_1$ intersects with just one clique, namely  $\Delta_2$.

Let $V(\Delta_1)=\{j_1,\ldots,{j_t}\}$  and $V(\Delta_2)=\{\ell_1,\ldots,{\ell_s}\}$ with $m\leq t,s$. We may assume that $V(\Delta_1)\sect V(\Delta_2)=\{k\}$ where $k=j_1=\ell_1$. Since $\Delta$ is closed, by Lemma~\ref{lasthope} we know that $x_{mk}$ is regular modulo $J_\Delta$. It follows from Lemma~\ref{rules} that $(S/J_\Delta)_{x_{mk}}\iso (S/L)_{x_{mk}}$ where $L=L_1+L_2+ \sum_{i=3}^rJ_{\Delta_i}$.  Here $L_1$ is generated by all $(m-1)$-minors of the matrix $X_{1\ldots m-1}^{j_2\ldots j_t}$, and $L_2$ is generated by all $(m-1)$-minors of the matrix $X_{1\ldots m-1}^{\ell_2\ldots \ell_s}$. The generators of $L_1$ are polynomials in a set of variables disjoint from those of $L'=L_2+ \sum_{i=3}^rJ_{\Delta_i}$. It is known  that $L_1$ is a prime ideal, see \cite[Theorem 7.3.1]{BH}. Thus $L$ is a prime ideal if and only $L'$ is a prime ideal. To see this observe that $(S/L')_{x_{mk}}\iso(S/J_{\Delta'})_{x_{mk}}$ where $\Delta'$ is the closed simplicial complex with clique decomposition  $\Delta'=\Delta_2\union\cdots\union \Delta_r$. By induction hypothesis, $J_{\Delta'}$ is a prime ideal. Hence $(S/L')_{x_{mk}}\iso (S/J_{\Delta'})_{x_{mk}}$ which implies that $(J_{\Delta'})_{x_{mk}}$ is a prime ideal. Since the generators of $L'$ are polynomials in variables different from  $x_{mk}$, it follows that $x_{mk}$ is regular modulo $L'$. Consequently $L'$ is a prime ideal.

(b) Due to Lemma~\ref{order} it suffices to show that $\Delta$ is closed  if each vertex of $G_\Delta$ has order at most $\min\{|V(\Delta_i)|,2\dim(\Delta)\}$.  We prove the assertion by induction on $r$. As before we assume that   $\Delta_1$ intersects with just one clique, namely  $\Delta_2$. By induction    it follows that $\Delta'=\Delta_2 \union\cdots \union \Delta_r$ is closed. Our assumption on the order of the vertices of $G_\Delta$ implies that $\Delta_2$ has at most $\min\{|V(\Delta_i)|,2\dim(\Delta)\}-1$ intersection points in $\Delta'$.

Hence among the vertices of $\Delta_2$ which are not intersection points in $\Delta'$ there is at least one which does not take all $m$ positions, say it misses the  $k^{\rm th}$ position. By symmetry we may assume this vertex is the intersection point with $\Delta_1$. Now we may label $\Delta_1$ such that the vertex in the intersection point does not have position $k$ for any facet of  $\Delta$ in $\Delta_1$.
\end{proof}

\begin{Theorem}
\label{cycle}
Let $\Delta$ be a simplicial complex such that $G_\Delta$ is a cycle. Then $J_\Delta$ is a prime ideal.
\end{Theorem}

\begin{proof}
Let $\Delta=\Delta_1\cup\ldots\cup \Delta_r$ be the clique decomposition of $\Delta$.
We consider the labeling on the vertices of $\Delta$ such that
\begin{eqnarray*}
&&V(\Delta_1)=\{1,2,\ldots,a_1\},
V(\Delta_2)=\{a_1,a_1+1,\ldots,a_2\},\ldots,
\\&&
V(\Delta_{r-1})=\{a_{r-2},a_{r-2}+1,\ldots,a_{r-1}\},
V(\Delta_r)=\{a_1-1,a_{r-1},a_{r-1}+1,\ldots,a_r\},
\end{eqnarray*}
where $1<a_1<\cdots<a_{r-1}<a_r=n$. Then $\Delta$ is closed with respect to the given labeling and, by Lemma~\ref{lasthope},  $x_{1a_1}$ is a regular element modulo $J_\Delta$.
It follows from Lemma~\ref{rules} that $(S/J_\Delta)_{x_{1a_1}}\iso (S/L)_{x_{1a_1}}$ where $L=L_1+L_2+ \sum_{i=3}^rJ_{\Delta_i}$.  Here $L_1$ is generated by all $(m-1)$-minors of the matrix $X_{2\ldots m}^{1\ldots a_1-1}$ and $L_2$  is generated by all $(m-1)$-minors of the matrix $X_{2\ldots m}^{a_1+1\ldots a_2}$.  Therefore, $J_\Delta$ is a prime ideal, if $L_{x_{1a_1}}$ is a prime ideal. Since the generators of $L$ are polynomials in variables different from $x_{1a_1}$, we conclude that $x_{1a_1}$ is regular modulo $L$. Hence $J_\Delta$ is a prime ideal if and only if $L$ is a prime ideal.

We first show that the generators of $L$ form a Gr\"obner basis for $L$.
In order to show this, note that the generators of $\sum_{i=3}^rJ_{\Delta_i}=J_{\Delta_3\union\cdots\union \Delta_r}$ form a Gr{\"o}bner basis for $\sum_{i=3}^rJ_{\Delta_i}$, since $\Delta_3\union\cdots\union \Delta_r$ is closed. Also
the generators of $L_1$ form a Gr{\"o}bner basis for $J_{\Gamma_1}$, where $\Gamma_1$ is the pure $(m-2)$-dimensional simplicial complex on the vertices $\{1,\ldots,a_1-1\}$,
and
the generators $L_2$ form a Gr{\"o}bner basis for $J_{\Gamma_2}$, where $\Gamma_2$ is the pure $(m-2)$-dimensional simplicial complex on the vertices $\{a_1+1,\ldots,a_2\}$. Finally, we note that the initial ideals of $L_1$, $L_2$ and, respectively,  $\sum_{i=3}^rJ_{\Delta_i}$, are minimally generated by monomials in pairwise disjoint sets of variables. Consequently, the generators of $L$ form indeed a Gr\"obner basis.

Next observe that the variable  $x_{m-1,a_1-1}$ does not appear in the support of the generators of $\ini_<(L)$. In particular, $x_{m-1,a_1-1}$ is regular modulo $L$.  By using Lemma~\ref{rules} we get
$(S/L)_{x_{m-1,a_1-1}}\iso (S/L'_1+L_2+L_r+\sum_{i=3}^{r-1}J_{\Delta_i})_{x_{m-1,a_1-1}}$, where
$L'_1$ is generated by all $(m-2)$-minors of the matrix
$X_{2\ldots m-2, m}^{1\ldots a_1-2}$, and $L_r$  is generated by all $(m-1)$-minors of the matrix $X_{1\ldots m-2, m}^{a_{r-1}\ldots a_r}$.

Since the generators of $L'=L'_1+L_2+L_r+\sum_{i=3}^{r-1}J_{\Delta_i}$ are polynomials in variables different from $x_{m-1,a_1-1}$, we conclude that $x_{m-1,a_1-1}$ is regular modulo $L'$. Hence $L_{x_{m-1,a_1-1}}$ is a prime ideal if and only if $L'$ is a prime ideal.

Since $L'_1$ is a prime ideal and the generators of $L_1'$ are polynomials in variables different from the variables of the other summands, in order to show that $L'$ is prime, it is enough to show that $C=L_2+L_r+\sum_{i=3}^{r-1}J_{\Delta_i}$ is a prime ideal.

We define the pure $(m-1)$-simplicial complex $\Delta'$ to be the simplicial complex with  clique decomposition
$\Delta'=\Delta_2\union\cdots\union \Delta_r$. Since the associated graph of $\Delta'$ is a tree we know from Theorem~\ref{tree} that $J_{\Delta'}$ is a prime ideal. Since  $(S/C)_{x_{1a_1}x_{m-1,a_1-1}}\iso(S/J_{\Delta'})_{x_{1a_1}x_{m-1,a_1-1}}$ and since ${x_{1a_1}}x_{m-1,a_1-1}$ is regular modulo $C$, the desired conclusion follows.
\end{proof}

The next result describes the case when each clique of $\Delta$ is a simplex.

\begin{Theorem}
\label{planar}
Let $\Delta$ be a  simplicial complex with graph $G_\Delta$ such that each clique of $\Delta$ is a simplex. Then the following holds:
\begin{enumerate}
\item[(a)] If $\Delta$ is closed, then $J_\Delta$ is  generated by a regular sequence;

\item[(b)] Given a graph $G$ and an integer $m\geq |V(G)|$, there exists a closed simplicial complex $\Delta$ with $G_\Delta=G$ such that each clique of $\Delta$ is a simplex of dimension $m-1$;

\item[(c)] $\Delta$ is closed, if $\dim \Delta+1$ is greater than or equal to the number of facets of $\Delta$.
\end{enumerate}
\end{Theorem}

\begin{proof}
(a) Let $\Delta=\Delta_1\cup\cdots\cup \Delta_r$ be the clique decomposition of $\Delta$. Since each clique is a simplex, it follows that $J_{\Delta_i}=(f_i)$ for all $i$,  where $f_i$ is a suitable $m$-minor, and $J_\Delta=(f_1,\ldots,f_r)$. Since $\Delta$ is closed the monomials $\ini_<(f_1), \ldots, \ini_<(f_r)$ are pairwise relatively prime. This implies that $f_1,\ldots,f_r$ is a regular sequence.

(b) We first assume that $m=|V(G)|$, and prove  in this case the assertion by induction on the number of vertices of $G$. The induction beginning is trivial. Now   assume that $|G|>1$, and choose a vertex $v$ of $G$. Let $G'$  be the induced subgraph on the vertices $V(G)\setminus \{v\}$. By induction there exists, for each $w\in V(G')$, a labeled simplex $\Delta'_w$  with $\dim \Delta'_w+1=  |V(G')|=|V(G)|-1$ such that the simplicial complex $\Delta'$ with clique decomposition $\Union_{w\in V(G')}\Delta'_w$ is closed and $G_{\Delta'}=G'$. We define new simplices $\Delta_w=\Delta'_w\union\{a_w\}$  where the labels $a_w$ are pairwise distinct and are bigger than all labels of $\Delta'$.

Let $w_1,\ldots, w_r$ be the neighbors of $v$ in $G$. Then we let $\Delta_v$ be the simplex whose  vertices are labeled by the  integers $a_{w_1},\ldots, a_{w_r}$ together with $|V(G)|-r$ numbers which are all bigger than all labels used in the construction so far.

Now let $m>|V(G)|$, and let $\Gamma$ be the closed  simplicial complex with $\dim  \Gamma= |V(G)|-1$, that we just have constructed. For each labeled simplex $\Gamma_i$ of dimension $|V(G)|-1$  of $\Gamma$ we define the new labeled simplex $\Delta_i=\Gamma_i\union\{b_{i1},\ldots, b_{is}\}$,  where $s=m-|V(G)|$ and where the numbers $b_{ij}$  are pairwise distinct and bigger than all labels of $\Gamma$. The simplicial complex $\Delta$ with facets $\Delta_i$ has the desired properties.

(c) Let  $\Delta$  be a simplicial  complex with graph $G_\Delta$ such that each clique of $\Delta$ is a simplex. Then, up to an isomorphism,  $\Delta$  is uniquely determined by $\dim \Delta$ and $G_\Delta$. Thus (b) is a simple consequence of (b).
\end{proof}

\begin{Corollary}
\label{complete}
Let $\Delta$ be a simplicial complex with graph  $G_\Delta$ such that each clique of $\Delta$ is a simplex of dimension $m-1$. Suppose that $G_\Delta$ is the  complete graph $K_r$.
Then  $\Delta$ is closed if and only if $m\geq r$.
\end{Corollary}

\begin{proof}
Each vertex of $K_r$ has order $r-1$. Therefore $m\geq r-1$, otherwise we could not associate the  graph $G_\Delta$ to $\Delta$. If $m=r-1$, then $\Delta$ has no free vertex, and hence $\Delta$ cannot be closed. On the other hand, if $m\geq r$, the assertion follows from Theorem~\ref{planar}.
\end{proof}

{}

\end{document}